\documentclass[12pt,reqno]{amsart}

\usepackage{amsmath}
\usepackage{amssymb}
\usepackage{amsthm}
\usepackage{mathtools}
\usepackage{enumitem}
\usepackage{graphicx}
\usepackage{tikz}
\usetikzlibrary{arrows.meta,calc,positioning}

\usepackage{microtype}
\usepackage{hyperref}

\makeatletter
\renewcommand{\subsection}{\@startsection{subsection}{2}{0pt}%
  {0.5\baselineskip}%
  {0.3\baselineskip}%
  {\normalfont\bfseries}}
\makeatother

\numberwithin{equation}{section}

\newtheorem{theorem}{Theorem}[section]
\newtheorem{lemma}[theorem]{Lemma}
\newtheorem{proposition}[theorem]{Proposition}

\theoremstyle{definition}
\newtheorem{definition}[theorem]{Definition}
\newtheorem{example}[theorem]{Example}

\usepackage[
    textwidth=6.25in,
    textheight=8.21in,
    centering
]{geometry}

\theoremstyle{definition}
\newtheorem{remark}[theorem]{Remark}

\newcommand{\YF}{\mathbb{YF}}
\newcommand{\PLab}{\operatorname{PLab}}
\newcommand{\arcjoin}{\mathrel{%
  \rule[0.1ex]{0.4pt}{1.1ex}%
  \rule[0.6ex]{1em}{0.4pt}%
  \rule[0.1ex]{0.4pt}{1.1ex}%
}}

\tikzset{
    okada diagram/.style={
    xscale=1.0,
    yscale=0.75,
    thick,
    baseline={(current bounding box.center)}
},
    okada label/.style={
        draw,
        fill=white,
        circle,
        inner sep=1pt,
        font=\scriptsize
    }
}

\newcommand{\OkadaBoundary}[1]{%
\foreach \i in {1,...,#1} {
\coordinate (L-\i) at (-1.6,{1.35*\i-1.35});
\coordinate (R-\i) at ( 1.6,{1.35*\i-1.35});
\node[left=2pt,font=\scriptsize] at (L-\i) {$\i$};
\node[right=2pt,font=\scriptsize] at (R-\i) {$\overline{\i}$};
    }%
}

\newcommand{\PropArc}[3]{%
    \draw (L-#1) -- (R-#2)
        node[pos=0.5,okada label] {#3};%
}

\newcommand{\LeftArc}[3]{%
    \draw (L-#1)
        .. controls +(0.75,0.35) and +(0.75,-0.35) ..
        (L-#2)
        node[pos=0.5,okada label] {#3};%
}

\newcommand{\RightArc}[3]{%
    \draw (R-#2)
        .. controls +(-0.75,-0.35) and +(-0.75,0.35) ..
        (R-#1)
        node[pos=0.5,okada label] {#3};%
}

\newcommand{\IdentitySix}{%
\begin{tikzpicture}[okada diagram]
    \OkadaBoundary{6}
    \foreach \i in {1,...,6}{
        \PropArc{\i}{\i}{\i}
    }
\end{tikzpicture}%
}

\newcommand{\GthreeSix}{%
\begin{tikzpicture}[okada diagram]
    \OkadaBoundary{6}
    \PropArc{1}{1}{1}
    \PropArc{2}{2}{2}
    \LeftArc{3}{4}{3}
    \RightArc{3}{4}{3}
    \PropArc{5}{5}{5}
    \PropArc{6}{6}{6}
\end{tikzpicture}%
}

\newcommand{\GfourSix}{%
\begin{tikzpicture}[okada diagram]
    \OkadaBoundary{6}
    \PropArc{1}{1}{1}
    \PropArc{2}{2}{2}
    \PropArc{3}{3}{3}
    \LeftArc{4}{5}{4}
    \RightArc{4}{5}{4}
    \PropArc{6}{6}{6}
\end{tikzpicture}%
}

\newcommand{\GthreeGfourSix}{%
\begin{tikzpicture}[okada diagram]
    \OkadaBoundary{6}
    \PropArc{1}{1}{1}
    \PropArc{2}{2}{2}
    \LeftArc{3}{4}{3}
    \PropArc{5}{3}{3}
    \RightArc{4}{5}{4}
    \PropArc{6}{6}{6}
\end{tikzpicture}%
}

\title{A Gelfand Model for the Okada Algebra}

\author{Harikrishnan T R}
\address{The Institute of Mathematical Sciences, Chennai, India}
\address{Homi Bhabha National Institute, Training School Complex, Anushakti Nagar, Mumbai 400094, India}
\email{haritr@imsc.res.in}


\begin{document}

\begin{abstract}
In this paper, we construct a Gelfand model for the Okada algebra $O_n(X,Y)$ with generic parameters $X$ and $Y$, on the space of symmetric Okada arc diagrams using a conjugation-type action. The model is constructed inductively by identifying the Okada algebra as a diagram algebra and using the Jones basic construction to obtain a tower of algebras that are themselves Okada algebras at lower levels. We use the model to obtain all the irreducible representations of $O_n(X,Y)$, indexed by the elements of rank $n$ of the Young--Fibonacci lattice, and identify them with the cell modules of $O_n(X,Y)$.
\end{abstract}
\maketitle
\bigskip
\noindent

\section*{Introduction}

The classical Robinson correspondence asserts a bijection between the symmetric group $S_n$ and
pairs of standard Young tableaux of the same shape $\lambda \vdash n$. In \cite{FominRSK}, Fomin generalized this correspondence to a broad class of posets known as
\emph{differential posets}, which were introduced independently by Stanley in \cite{StanleyDifferential}. The classical example of a differential poset other than the Young lattice is the Young--Fibonacci lattice, and these two are the only $1$-differential posets which are lattices.
Fomin's generalized Robinson correspondence for differential posets gives a bijection between the symmetric group $S_n$ and pairs of saturated chains in the poset starting from the minimum element $0$ and ending at the same vertex at level $n$, using local growth rules.

A Gelfand model for a finite-dimensional semisimple algebra $A$ is a representation of $A$ which contains each irreducible representation of $A$ with multiplicity one.
By the Robinson correspondence, we see that the number of involutions in $S_n$ is the number of paths in the Young lattice starting from $0$ and ending at a partition of $n$. By basic representation theory of symmetric groups, this number is the sum of the dimensions of the irreducible representations of $S_n$. This suggests that there might exist a Gelfand model for the symmetric group on the space with basis the involutions in the group, with an appropriate action. Such models have been constructed for the symmetric group on the space of involutions with a signed conjugation action in \cite{KodiyalamVerma2004} and \cite{AdinPostnikovRoichman2008}.

In 1994, Okada introduced a tower of semisimple algebras $O_n(X,Y)$ with $X$ and $Y$ generic parameters, whose Bratteli diagram is the Young--Fibonacci lattice \cite{Okada}, similar to the Young lattice for the symmetric groups. He also constructed all the irreducible representations of $O_n(X,Y)$ and showed that they are indexed by the rank-$n$ elements of the Young--Fibonacci lattice. As in the symmetric group case, the number of involutions in $S_n$ again equals the sum of the dimensions of the irreducible representations of the $n$th Okada algebra $O_n(X,Y)$. This suggests the possibility of a Gelfand model for the Okada algebra on the space of involutions in $S_n$.

In \cite{Hivert-Scott,HivertScott}, Hivert and Scott realized the Okada algebra as a diagram algebra inspired by Viennot's theory of heaps of dimers. They do this by defining Okada arc diagrams, implementing the Okada relations on them and then showing that the Okada algebra is isomorphic to the diagram algebra generated by these arc diagrams. 
In \cite{Mazorchuk}, Mazorchuk constructed Gelfand models for some diagram algebras including the Brauer algebra and the Temperley--Lieb algebra.
In the 2015 paper \cite{HalversonReeks}, Halverson and Reeks constructed a Gelfand model for the partition algebra and some of its subalgebras, including the Brauer algebra, the Temperley--Lieb algebra and the planar partition algebra. They do this by identifying these algebras as diagram algebras and using the Jones basic construction to lift a Gelfand model from smaller algebras in the tower to the whole algebra.

For generic $X$ and $Y$, the Okada algebras $\{O_n(X,Y)\}_{n\geq 1}$ form a tower of split semisimple algebras, with $\dim O_n(X,Y)=n!$. Set $O_n:= O_n(X,Y)$. We construct an $O_n$-module $M_n$ on the space of symmetric Okada arc diagrams using a conjugation-type action:
\begin{equation*}
 d\cdot t=
\begin{cases}
\displaystyle \lambda(d,t)[d\circ t\circ d^*],
& \text{if } \PLab([d\circ t\circ d^*])=S,\\[6pt]
0,
& \text{otherwise}
\end{cases}
\end{equation*} where $d$ and $t$ are basis arc diagrams with $t$ symmetric, $\PLab(t)$ the Fibonacci set $S$, $\lambda(d,t)$ a scalar, $d^*$ the mirror image of $d$, $\PLab$ the propagating label set of an arc diagram and $[d\circ t\circ d^*]$ the monoid product of the arc diagrams. This module could equally be regarded as a module on the space of involutions in $S_n$ since there is a bijection between the set of involutions and the set of symmetric Okada arc diagrams of rank $n$ by \cite[Theorem 4.14]{HivertScott}.

Inspired by \cite{HalversonReeks}, we use the Jones basic construction to show that $M_n$ is a Gelfand model for $O_n$. We do this by identifying particular idempotents $e_n$ in $O_n(X,Y)$ (following \cite{HalversonRam}) such that there is a bijection between the irreducible representations of the ideal $J_n=O_ne_nO_n$ and those of $O_{n-2}$. The complementary subalgebra $C_n$ of $J_n$ in $O_n$ is isomorphic to $O_{n-1}$. This gives the recursion \[
\Lambda_{O_n}=\Lambda_{O_{n-2}}\sqcup\Lambda_{O_{n-1}},\]
where $\Lambda_{O_n}$ denotes an indexing set of irreducibles for $O_n$, which is the recursion satisfied by the rank sizes of the Young--Fibonacci lattice. This suggests that the irreducible representations of $O_n$ can be constructed inductively.

We see that the $O_n$-module $M_n$ defined as the direct sum of the submodules $M_n^S$ for each Fibonacci set $S$ of rank $n$ with the action defined above is a Gelfand model by showing that any two of the submodules have no common irreducibles and their characters satisfy certain relations; see \eqref{character relations}. Then the proof that $M_n$ is a Gelfand model is done at the character level since we work with split semisimple algebras. Later, we use the model to explicitly identify all the irreducible representations of the Okada algebra and to identify them with the simple cell modules of the algebra.

This model has some advantages besides being a natural construction. The proof proceeds through the Jones basic construction, via the ideal $J_n=O_ne_nO_n$ and the complementary subalgebra $C_n\cong O_{n-1}$, and the construction of the Gelfand model is inductive. We also obtain all the irreducible representations of $O_n$ without any prior information about the irreducibles beyond their dimensions, which are used in showing the split semisimplicity of the algebra. Apart from the information about the irreducible representations used to establish split semisimplicity, our construction does not use Okada's explicit realization of the irreducible modules. Although our construction and Okada's original construction are realized on the same underlying space (under the identification of a path with a symmetric Okada arc diagram), the resulting actions are quite different. In Okada's construction, a generator generally sends a basis element to a nontrivial linear combination of basis elements. In contrast, our conjugation-type action sends each basis element to a scalar multiple of a single basis element (possibly zero).

Sections 1--3 recall the Young--Fibonacci lattice and the diagram realization; Sections 4--6 construct the model; Section 7 identifies the irreducibles, and relates them to the cell modules. 
 
\section{The Young--Fibonacci lattice}

In \cite{StanleyDifferential}, Stanley defines the Young--Fibonacci lattice as a poset with vertices the finite words in the alphabet $\{1,2\}$. 
The \emph{Young--Fibonacci lattice} $\YF$ is the poset with covering relation:
$s\lessdot w$ if $s$ is obtained from $w$ by removing the first letter $1$ or
replacing any one $2$ occurring before the first occurrence of $1$ by $1$. We get the elements of rank $n$ (sum of letters is $n$) recursively as follows: the set of elements of rank $n$ is the union of elements obtained by prepending 2 to the elements of rank $(n-2)$ or prepending 1 to the elements of rank $(n-1)$. Throughout this paper, we use the convention $
F_1=1,F_2=2, F_n=F_{n-1}+F_{n-2}\quad (n\geq 3),$
so that the Fibonacci sequence is $1,2,3,5,8,\ldots.$
Thus the number of elements of rank $n$ of $\YF$ is the $n$th Fibonacci number $F_n$.
The Young--Fibonacci lattice is a graded graph with the grading the sum of letters of each word; see Figure \ref{YF1}.
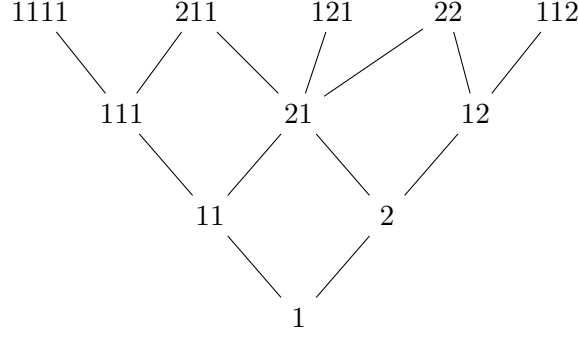
\begin{figure}[ht]
\centering
\begin{tikzpicture}[scale=0.9, every node/.style={font=\small}]
\node (n1) at (0,1.5) {$1$};
\node (n2) at (-1.3,3) {$11$};     \node (n11) at (1.3,3) {$2$};
\node (n12) at (-2.6,4.5) {$111$}; \node (n21) at (0,4.5) {$21$};
\node (n111) at (2.6,4.5) {$12$};
\node (n112) at (-3.8,6) {$1111$}; \node (n22) at (-1.5,6) {$211$};
\node (n121) at (0.5,6) {$121$};  \node (n211) at (2.2,6) {$22$};
\node (n1111) at (3.8,6) {$112$};
\draw (n1)--(n2); \draw (n1)--(n11);
\draw (n2)--(n12); \draw (n2)--(n21);
\draw (n11)--(n21); \draw (n11)--(n111);
\draw (n12)--(n112); \draw (n12)--(n22);
\draw (n21)--(n22); \draw (n21)--(n121); \draw (n21)--(n211);
\draw (n111)--(n211); \draw (n111)--(n1111);
\end{tikzpicture}
\caption{The Young--Fibonacci lattice $\YF$ (ranks $1$--$4$).}
\label{YF1}
\end{figure}

 In this paper we will use an equivalent definition of the Young--Fibonacci lattice in terms of Fibonacci sets given by Hivert and Scott in \cite{Hivert-Scott}. 
\begin{definition}
    For a positive integer $n$, a \emph{Fibonacci set of rank $n$} is a subset $S=\{s_1<s_2<\cdots<s_k\}$ of $[n]:=\{1,2,\ldots,n\}$ such that $n \equiv k\pmod 2 $ and $s_i \equiv i \pmod 2$ for $1 \leq i \leq k$. We denote the set of Fibonacci sets of rank $n$ by $\YF_n$.
\end{definition}

\begin{definition}
    The \emph{Young--Fibonacci lattice} is the graded poset $\YF=\bigsqcup_{n\geq 1} \YF_n$ with the covering relation $S \lessdot T$ if and only if $S \in \YF_{n-1}$ and $T \in \YF_n$ for some $n\geq 2$ and one set is obtained from the other by removing its largest element.
\end{definition}
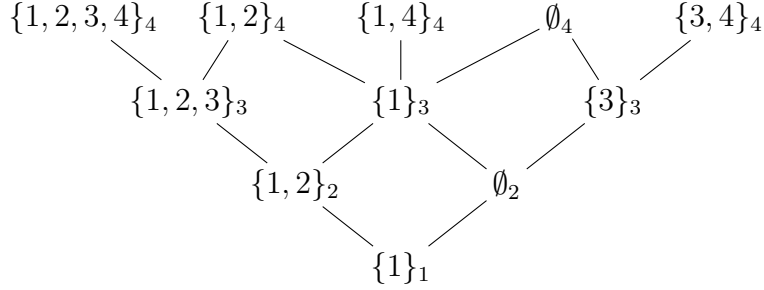
\begin{figure}[ht]
\centering

\begin{tikzpicture}[
    xscale=1.4,
    yscale=1.1,
    every node/.style={inner sep=2pt}
]

\node (n1) at (0,1) {$\{1\}_1$};

\node (n2a) at (-1,2) {$\{1,2\}_2$};
\node (n2b) at (1,2) {$\emptyset_2$};

\node (n3a) at (-2,3) {$\{1,2,3\}_3$};
\node (n3b) at (0,3) {$\{1\}_3$};
\node (n3c) at (2,3) {$\{3\}_3$};

\node (n4a) at (-3,4) {$\{1,2,3,4\}_4$};
\node (n4b) at (-1.5,4) {$\{1,2\}_4$};
\node (n4c) at (0,4) {$\{1,4\}_4$};
\node (n4d) at (1.5,4) {$\emptyset_4$};
\node (n4e) at (3,4) {$\{3,4\}_4$};

\draw (n1) -- (n2a);
\draw (n1) -- (n2b);

\draw (n2a) -- (n3a);
\draw (n2a) -- (n3b);
\draw (n2b) -- (n3b);
\draw (n2b) -- (n3c);

\draw (n3a) -- (n4a);
\draw (n3a) -- (n4b);

\draw (n3b) -- (n4b);
\draw (n3b) -- (n4c);
\draw (n3b) -- (n4d);

\draw (n3c) -- (n4d);
\draw (n3c) -- (n4e);

\end{tikzpicture}

\caption{The Young--Fibonacci lattice up to rank $4$.}
\label{YF2}
\end{figure} We say that a Fibonacci set $S$ is of rank $n$ if $S \in \YF_n$; see Figure \ref{YF2}.
Note that the subscripts in the vertices of the Young--Fibonacci lattice indicate the rank of the vertex.

\begin{remark}\label{YF equivalence}
    The two definitions of the Young--Fibonacci lattice are equivalent under the identification of a binary word with the set of the sums of its suffixes whose first digit is 1; see \cite[Propositions 4.8 and 4.9]{HivertScott}.
\end{remark}

\section{The Okada algebra as a diagram algebra}

In this section we quickly review the realization of the Okada algebra as a
diagram algebra. We refer the reader to \cite{Hivert-Scott,HivertScott} for
details. In \cite{Okada}, Okada introduced a tower of algebras
$\{O_n(X,Y)\}_{n\ge1}$ whose Bratteli diagram is the Young--Fibonacci lattice.
This construction parallels the tower of symmetric group algebras whose Bratteli diagram is the Young lattice.

\begin{definition}
Let $K$ be a field of characteristic $0$ and let
$X=(x_1,x_2,\ldots)$ and $Y=(y_1,y_2,\ldots)$ be two sequences of parameters
in $K$. For $n\ge1$, the \emph{Okada algebra} $O_n(X,Y)$ is the associative
unital $K$-algebra generated by $E_1,E_2,\ldots,E_{n-1}$ subject to the
relations
\begin{enumerate}
\item $E_i^2=x_iE_i$ for $1\le i\le n-1$,
\item $E_iE_j=E_jE_i$ for $|i-j|\ge 2$,
\item $E_{i+1}E_iE_{i+1}=y_iE_{i+1}$ for $1\le i\le n-2$.
\end{enumerate}
Only the parameters $x_1,\ldots,x_{n-1}$ and $y_1,\ldots,y_{n-2}$ occur in
$O_n(X,Y)$.
\end{definition}

When we consider the \emph{generic} Okada algebra, we take the parameters $x_i$
and $y_i$ to be algebraically independent indeterminates over a field $K_0$ of
characteristic $0$ and work over the corresponding rational function field $K=K_0(x_1,\ldots,y_1, \ldots)$.
Unless otherwise specified, statements concerning semisimplicity and
irreducible representations are understood in this generic setting.

Okada showed that in this setting $O_n(X,Y)$ is semisimple of dimension $n!$,
and that its irreducible representations are indexed by the vertices of rank
$n$ of the Young--Fibonacci lattice \cite[Theorem 2.6]{Okada}.

Now we describe the realization of the Okada algebra as a diagram algebra.
We endow the set $\{1,2,\ldots,n\}\cup\{\overline1,\overline2,\ldots,\overline n\}$ with the
total order $1<2<\cdots<n<\overline n<\cdots<\overline2<\overline1$. We identify $\overline i$
with $-i$ for $1\le i\le n$, and for
$a\in\{1,\ldots,n,\overline n,\ldots,\overline1\}$ we write $|a|$ for its absolute
value; thus $|i|=|\overline i|=i$.

\begin{definition}
A rank $n$ \emph{Okada arc diagram} is a noncrossing perfect matching linking
the boundary vertices $\{1,\ldots,n\}$ and $\{\overline1,\ldots,\overline n\}$ in which
every arc $a\arcjoin b$ is assigned a height label
$h(a\arcjoin b)\in\mathbb{N}$ satisfying
\begin{enumerate}
\item $1\le h(a\arcjoin b)\le\min(|a|,|b|)$,
\item $h(a\arcjoin b)\equiv\min(|a|,|b|)\pmod 2$,
\item if the arc $a\arcjoin b$ is nested inside the arc $c\arcjoin d$, then
$h(a\arcjoin b)>h(c\arcjoin d)$, where we say that $a\arcjoin b$ is nested
inside $c\arcjoin d$ if $c<a<b<d$.
\end{enumerate}
See Figure \ref{fig:okada-examples} for examples.
An arc joining vertices on opposite sides is called \emph{propagating}.
\end{definition}
We use the notation $h(a\arcjoin b)=i$ to denote an arc $a\arcjoin b$ with height label $i$. 
\begin{definition}
Let $d_1$ and $d_2$ be Okada arc diagrams of rank $n$. The \emph{composition}
$d_1\circ d_2$ is obtained by placing $d_2$ to the right of $d_1$, identifying
the vertex $\overline i$ on the right boundary of $d_1$ with the vertex $i$ on the
left boundary of $d_2$ for $1\le i\le n$, and concatenating the corresponding
arcs.

The composition $d_1\circ d_2$ may contain closed loops. Each arc or loop in
$d_1\circ d_2$ is assigned the minimum of the height labels of the arc
fragments of $d_1$ and $d_2$ from which it is formed. We denote by
$[d_1\circ d_2]$ the diagram obtained from $d_1\circ d_2$ by deleting all
closed loops.
\end{definition}

We denote the set of Okada arc diagrams of rank $n$ by $\mathcal D_n$. With the
multiplication $(d_1,d_2)\mapsto[d_1\circ d_2]$, the set $\mathcal D_n$ forms a
monoid with identity element the diagram with $n$ propagating arcs
$h(i\arcjoin\overline i)=i$, $i=1,\ldots,n$; see
\cite[Theorem 3.43]{HivertScott}.

\begin{definition}
For $1\le i\le n-1$, define the \emph{elementary Okada arc diagram} $G_i$ to be
the Okada arc diagram of rank $n$ with the propagating arcs $j\arcjoin\overline j$
for $j\ne i,i+1$ with height labels $j$, and the non-propagating arcs
$i\arcjoin i+1$ and $\overline i\arcjoin\overline{i+1}$ with height label
$h(i\arcjoin i+1)=h(\overline i\arcjoin\overline{i+1})=i$.
\end{definition}

The diagrams $G_1,\ldots,G_{n-1}$ generate the monoid $\mathcal D_n$; see
\cite[Corollary 3.59]{HivertScott}.
The diagram product can be extended to incorporate the parameters $X$ and $Y$.
For $d_1,d_2\in\mathcal D_n$, let $\lambda(d_1,d_2)\in K$ denote the monomial in
the parameters $X$ and $Y$ associated to the concatenation $d_1\circ d_2$, as
described in \cite[Corollary 6.15]{HivertScott}.

\begin{definition}
Let $\widetilde O_n(X,Y)$ be the vector space over $K$ with basis
$\mathcal D_n$. Define multiplication on basis elements by
\[
d_1d_2=\lambda(d_1,d_2)\,[d_1\circ d_2],
\qquad d_1,d_2\in\mathcal D_n,
\]
and extend it $K$-bilinearly.
\end{definition}

For $d\in\mathcal D_n$, let $E_d\in O_n(X,Y)$ denote the monomial in the
generators associated to $d$ in \cite[Definition 6.1]{HivertScott}. By
\cite[Proposition 6.2]{HivertScott}, the map $d\mapsto E_d$ extends to a linear
isomorphism $\widetilde O_n(X,Y)\to O_n(X,Y)$ carrying
$d_1d_2=\lambda(d_1,d_2)[d_1\circ d_2]$ to
$E_{d_1}E_{d_2}=\lambda(d_1,d_2)E_{[d_1\circ d_2]}$; in particular, the
multiplication on $\widetilde O_n(X,Y)$ is associative and unital, and
\[
\widetilde O_n(X,Y)\cong O_n(X,Y).
\]
Under this isomorphism the generator $E_i$ corresponds to the elementary Okada
arc diagram $G_i$ \cite[Corollary 3.61]{HivertScott}. Thus, from now on, we
identify $O_n(X,Y)$ with its diagram realization $\widetilde O_n(X,Y)$ and
regard the Okada arc diagrams of rank $n$ as a basis of $O_n(X,Y)$.

\begin{figure}[ht]
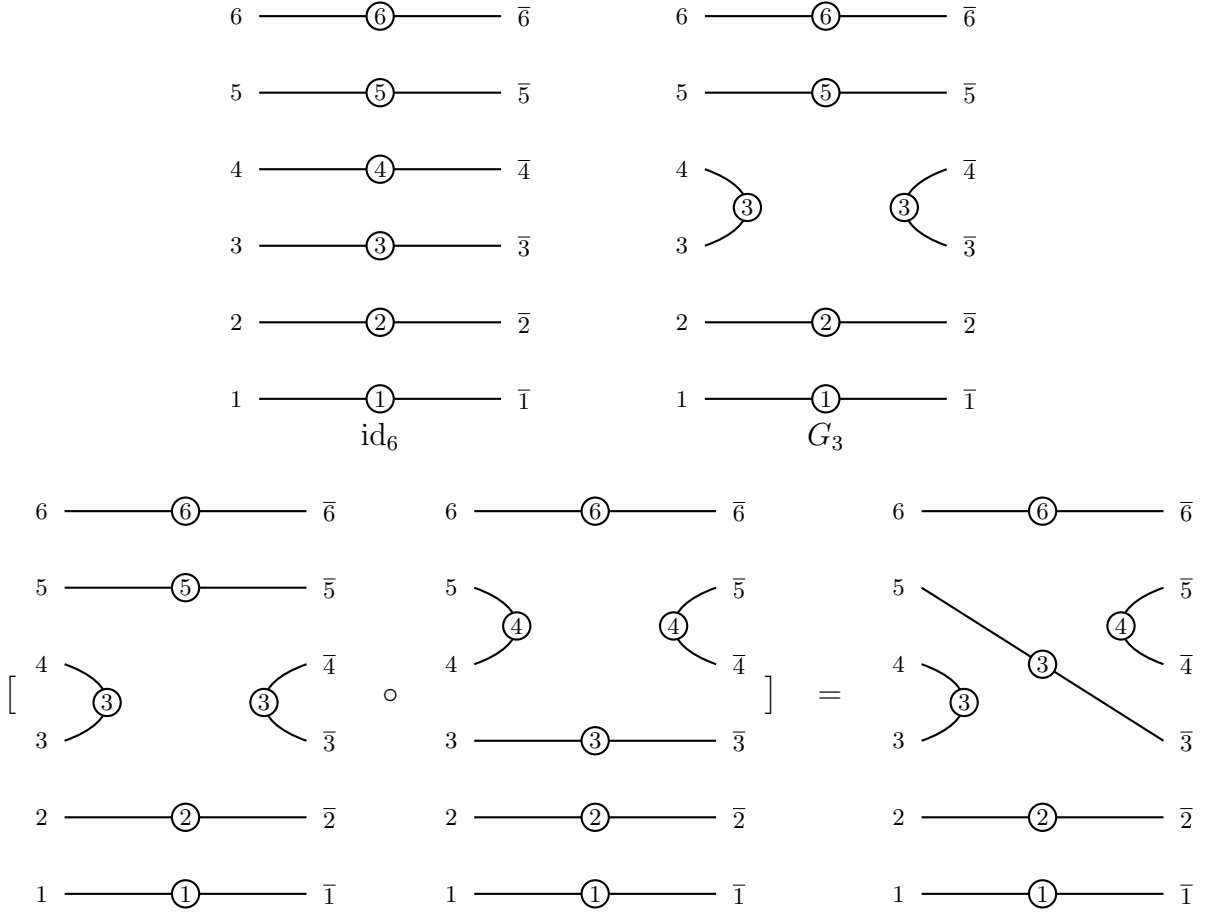

\centering
\[
\begin{array}{c@{\qquad\qquad}c}
\IdentitySix & \GthreeSix \\[-0.01mm]
\mathrm{id}_6 & G_3
\end{array}
\]
\vspace{0.5em}
\[
\bigl[\;\GthreeSix
\quad\circ\quad
\GfourSix\;\bigr]
\quad=\quad
\GthreeGfourSix
\]
\caption{The identity diagram, the elementary diagram $G_3$, and the product
$[G_3\circ G_4]$ in rank $6$.}
\label{fig:okada-examples}
\end{figure}

For $d'\in\mathcal D_{n-1}$, let $\iota_{n-1}(d')\in\mathcal D_n$ denote the
Okada arc diagram obtained from $d'$ by adding the labelled propagating arc
$h(n\arcjoin\overline n)=n$. This defines an injective monoid morphism
$\iota_{n-1}\colon\mathcal D_{n-1}\to\mathcal D_n$
\cite[Remark 3.48]{HivertScott}.

\begin{proposition}[{\cite[Propositions~3.56 and 3.57]{HivertScott}}]\label{J_n}
Let $d$ be an Okada arc diagram of rank $n$ which does not contain the arc
$h(n\arcjoin\overline n)=n$, and let $i$ be the largest integer such that $d$
contains the arc $h(\overline i\arcjoin\overline{i+1})=i$; such an $i$ exists by
\cite[Corollary 3.53]{HivertScott}. Then there exists a unique Okada arc
diagram $d'$ of rank $n-1$ such that
\[
d=\bigl[\iota_{n-1}(d')\circ G_{n-1}\circ G_{n-2}\circ\cdots\circ G_i\bigr].
\]
\end{proposition}

\begin{remark}
For generic parameters $X$ and $Y$, the Okada algebras $\{O_n(X,Y)\}_{n\ge1}$
form a tower of finite-dimensional semisimple algebras with
$\dim_K O_n(X,Y)=n!$ \cite[Theorem 2.6]{Okada}, having as basis the Okada arc
diagrams of rank $n$ \cite[Proposition 6.2]{HivertScott}. We regard
$O_{n-1}(X,Y)\subset O_n(X,Y)$ via $\iota_{n-1}$.
\end{remark}
\begin{definition}
Let $K$ be a field. A finite-dimensional $K$-algebra $A$ is called
\emph{split semisimple} if
\[
A\cong\bigoplus_{i=1}^r M_{d_i}(K)
\]
as $K$-algebras, for some positive integers $r,d_1,\ldots,d_r$.
\end{definition}
\begin{proposition}\label{split}
The Okada algebra $O_n(X,Y)$ is split semisimple.
\end{proposition}

\begin{proof}
The algebra $O_n(X,Y)$ is semisimple, and its
irreducible modules are indexed by $v\in \YF_n$. Let $V_v$ denote the
irreducible module corresponding to $v$. Then
$\dim_K V_v=F_v$, where $F_v$ is the number of paths in $\YF$ from $\{1\}_1$ to
$v$ by \cite[Theorem 2.6]{Okada}. Moreover, Fomin's generalized Robinson correspondence for differential posets gives a bijection between the symmetric group $S_n$ and pairs of saturated chains in $\YF$ starting from $1$ and ending at the same vertex at rank $n$; see \cite[Example 2.6.8]{Roby1991}. Therefore, we have
$\sum_{v\in\YF_n} F_v^2=n!.$
For each $v\in\YF_n$, let
$D_v=\operatorname{End}_{O_n(X,Y)}(V_v)$. As a consequence of the Artin--Wedderburn theorem,
\[
\dim_K O_n(X,Y)
 =\sum_{v\in\YF_n}
 \frac{(\dim_K V_v)^2}{\dim_K D_v}.
\]
Since $\dim_K O_n(X,Y)=n!$, we obtain
\[
n!
=\sum_{v\in\YF_n}\frac{F_v^2}{\dim_K D_v}
=\sum_{v\in\YF_n}F_v^2.
\]
Since $\dim_K D_v\geq 1$ for every $v$, and $F_v>0$, this equality forces
$\dim_K D_v=1$ for every $v$. Hence $D_v\cong K$ for all $v\in\YF_n$.
Therefore every simple component of $O_n(X,Y)$ is a full matrix algebra
over $K$, and hence $O_n(X,Y)$ is split semisimple.
\end{proof}
\section{Okada half arc diagrams and propagating label sets}

In this section, we briefly recall the definitions and results that will be used frequently throughout the paper. We refer the reader to \cite{HivertScott} for further details.

\begin{definition}
An \emph{Okada half arc diagram} of rank $n$ is obtained by cutting an
Okada arc diagram vertically through its middle. It consists of the
vertices $1,\ldots,n$, together with labelled full arcs joining two
vertices and labelled half-arcs having one free endpoint. A half-arc, that is, an arc with one free endpoint, is called \emph{propagating}.
\end{definition}

\begin{definition}
For an Okada arc diagram $d$, the \emph{bra} $|d \rangle$ is the half
arc diagram obtained by restricting $d$ to its left half. The
\emph{ket} $\langle d|$ is defined by
$
\langle d| =| d^*\rangle,$
where $d^*$ denotes the \emph{mirror} image of $d$, obtained by reflecting $d$ about the vertical axis. The diagram $d$ is called \emph{symmetric} if $d=d^*$.
\end{definition}
Note that $[d_1 \circ d_2]^*=[d_2^* \circ d_1^*]$ for Okada arc diagrams $d_1$ and $d_2$ by \cite[Remark 3.47]{HivertScott}. 
\begin{definition}
The \emph{propagating label set} of an Okada half arc diagram $H$ of rank $n$ is the subset
$\PLab(H)\subseteq [n]$
consisting of the height labels of the propagating half-arcs of $H$. By \cite[Lemma 4.4]{HivertScott}, $\PLab(\langle d|)=\PLab(|d\rangle)$ for an Okada arc diagram $d$. Therefore we can define
\[
\PLab(d):=\PLab(\langle d|)=\PLab(|d\rangle).
\]
\end{definition}
By \cite[Proposition 4.7]{HivertScott}, the propagating label set of an Okada arc diagram $d$ of rank $n$ is a Fibonacci set of rank $n$.
The following partial order on the set of Fibonacci sets of rank $n$ and Proposition \ref{structure} allow us to describe the structure of the Okada monoid.
\begin{definition} 
    For two Fibonacci sets $S=\{s_1< \cdots< s_k\}$ and $T=\{t_1< \cdots< t_\ell\}$ of rank $n$, we say that $T$ dominates $S$ and write $S \preccurlyeq T$ if $k \leq \ell$ and $s_{k-i} \leq t_{\ell-i}, ~i=0, \ldots, k-1;$ and $S \prec T$ if $S \preccurlyeq T$ and $S \neq T$.
\end{definition} 
Under the above order $\preccurlyeq$, $\YF_n$ becomes a ranked distributive lattice; see \cite[Proposition 5.2]{HivertScott}.
\begin{proposition}[{\cite[Lemma 5.11 and Corollary 5.12]{HivertScott}}]\label{structure}
    Let $e$ and $f$ be two Okada arc diagrams of rank $n$. Then either $| [e \circ f]\rangle = | e \rangle$ and thus $\PLab([e \circ f])= \PLab(e)$ or $\PLab([e \circ f]) \prec \PLab(e)$. Therefore, $\PLab([e \circ f ]) \preccurlyeq \inf (\PLab(e), \PLab(f)).$
\end{proposition}
\begin{lemma}\label{halves}
Let $t$ be a symmetric Okada arc diagram of rank $n$ with $\PLab(t)=S$ and $d\in\mathcal D_n$. Then
$\PLab([d\circ t\circ d^*])=S$ if and only if $\PLab([d\circ t])=S$, and in
that case $\bigl|[d\circ t\circ d^*]\bigr\rangle=\bigl|[d\circ t]\bigr\rangle$.
\end{lemma}
\begin{proof}
   If $\PLab([ d \circ t \circ d^*])=S$, then by Proposition~\ref{structure}, we have $\PLab([d \circ t])=S$. For the other direction, suppose that $\PLab([d \circ t])=S$. Since $t =t^*$, and taking the mirror image does not change the propagating label set, $\PLab([t \circ d^*])=S$. Then by Proposition \ref{structure}, $| [t \circ d^*] \rangle = | t \rangle$. Since they have the same left half, composing on the left with $d$ does not change the propagating label set of the new diagrams, i.e., $\PLab([d \circ t \circ d^*])=\PLab([d \circ t])$. Again by Proposition \ref{structure}, $| [d \circ t \circ d^*] \rangle = | [d \circ t] \rangle$.
\end{proof}
Now we record some results concerning the monomial $\lambda(d,t)$ arising from the product $dt$ of two Okada arc diagrams $d$ and $t$ of rank $n$. For an Okada arc diagram $d$ of rank $n$, let $\mathrm{Length}(d)$ denote its length in the sense of \cite[Proposition~5.27]{HivertScott}; by that proposition, 
\[ 
\mathrm{Length}(d)= \frac{n(n+1)}{2}- \sum \ell,
\] where the sum runs over all arcs $h(i \arcjoin j)=\ell$ of $d$.
\begin{lemma}\label{properties of lambda}
\begin{enumerate}
\item Let $a,b,c\in\mathcal D_n$ with $|b\rangle=|c\rangle$. Then
$\lambda(a,b)=\lambda(a,c)$.
\item For $a\in\mathcal D_{n-2}$, regarded as an element of $\mathcal D_n$,
$\mathrm{Length}([a\circ G_{n-1}])=\mathrm{Length}(a)+\mathrm{Length}(G_{n-1})$,
and $\lambda(a,G_{n-1})=1$.
\end{enumerate}
\end{lemma}
\begin{proof}
    Part (1) follows from \cite[Remark 6.4]{HivertScott}. We now prove Part (2).
By the definition of $G_{n-1}$, $\mathrm{Length}(G_{n-1})=1$. The diagram $[a \circ G_{n-1}]$ is obtained from $a$ by replacing the propagating arcs $h( n-1 \arcjoin \overline{n-1})=n-1$ and $h(n \arcjoin \overline{n})=n$ by the non-propagating arcs $h(n-1 \arcjoin n)=n-1$ and $h(\overline{n-1} \arcjoin \overline{n})=n-1$ with all other arcs and their height labels unchanged.  Let $L$ be the sum of the height labels of the arcs of $a$ as a diagram of rank $n-2$, and set $k:= \frac{n(n+1)}{2}-L$. Then 
\begin{align*}
\mathrm{Length}(a)
  &= k-\bigl(n+(n-1)\bigr) = k-2n+1,\\
\mathrm{Length}\bigl([a\circ G_{n-1}]\bigr)
  &= k-2(n-1) = k-2n+2
   = \mathrm{Length}(a)+\mathrm{Length}(G_{n-1}).
\end{align*}
Since this product is length-additive, \cite[Proposition 6.2]{HivertScott} gives
$\lambda(a,G_{n-1})=1$.
\end{proof}

\section{Construction of the Gelfand model}
In this section we construct a representation of the Okada algebra $O_n(X,Y)$ on the space of symmetric Okada arc diagrams using a conjugation-type action and state our main theorem giving a Gelfand model for the Okada algebra $O_n(X,Y)$ for generic parameters $X$ and $Y$. Unless otherwise specified, throughout this paper $O_n$ denotes the Okada algebra $O_n(X,Y)$ with $X,Y$ generic. 
\begin{definition}
Let $n$ be a positive integer. Define the set $I_n$ of symmetric Okada arc
diagrams of rank $n$ by
\[
I_n=\{t\in\mathcal{D}_n : t=t^*\},
\]
and set $M_n=KI_n$.
\end{definition}
Now we define the action of Okada arc diagrams of rank $n$ on $M_n$ and extend it linearly to $O_n(X,Y)$. Let $d$ be an Okada arc diagram of rank $n$ and let $t\in I_n$. Define

\begin{equation}\label{action}
 d\cdot t=
\begin{cases}
\displaystyle \lambda(d,t)[d\circ t\circ d^*],
& \text{if } \PLab([d\circ t\circ d^*])=\PLab(t),\\[6pt]
0,
& \text{otherwise}.
\end{cases}
\end{equation}
Recall that $[d\circ t\circ d^*]$ denotes the Okada arc diagram obtained after composing, deleting loops, labelling the remaining arcs by their heights, and taking isotopy class. Note that $[d \circ t \circ d^*]$ is symmetric since $t$ is symmetric. The scalar $\lambda(d,t)$ denotes the coefficient of $[d \circ t ]$ in the product $dt$.
 
Recall that a \emph{Gelfand model} for a finite-dimensional split semisimple algebra $A$ is an $A$-module in which every irreducible $A$-module occurs exactly once, up to isomorphism. Now we state the main theorem of this paper which gives a Gelfand model for the Okada algebras.
\begin{theorem}\label{main}
Let $M_n=KI_n$ be as above. Then the following hold.
\begin{enumerate}
\item Equation \eqref{action} defines an $O_n(X,Y)$-module structure on $M_n$.
\item The module $M_n$ is a Gelfand model for the Okada algebra $O_n(X,Y)$.
\end{enumerate}
\end{theorem}
\begin{proof}[Proof of part (1)]
 It is obvious that $\mathrm{id}_n \cdot t=t,~ t \in I_n.$
    We need to show that $(d_1d_2)\cdot t=d_1\cdot(d_2\cdot t)$ for $t \in I_n,~ d_1,d_2 \in \mathcal{D}_n.$ Let $\PLab(t)=S$. Assume $\PLab([d_2\circ t\circ d_2^*])\neq S$. Then by Proposition \ref{structure}, $\PLab([d_2\circ t\circ d_2^*])\prec \PLab(t)=S$. Hence $d_2\cdot t=0$, and therefore $d_1\cdot(d_2\cdot t)=0$.
In this case, $\PLab([[d_1\circ d_2]\circ t\circ [d_1\circ d_2]^*])\preccurlyeq \PLab([d_2\circ t\circ d_2^*])\prec S$. Thus $(d_1d_2)\cdot t=0$. Hence the first zero case is done.

Now suppose $\PLab([d_2\circ t\circ d_2^*])=S$ and $\PLab([[d_1\circ d_2]\circ t\circ [d_1\circ d_2]^*])\neq S$. Then $(d_1d_2)\cdot t=0$.
At the same time, $d_2\cdot t=\lambda(d_2,t)[d_2\circ t\circ d_2^*]$. Therefore
\[
d_1\cdot(d_2\cdot t)
=
\lambda(d_2,t)
d_1\cdot [d_2\circ t\circ d_2^*].
\]
But $\PLab([d_1\circ [d_2\circ t\circ d_2^*]\circ d_1^*])=\PLab([[d_1\circ d_2]\circ t\circ [d_1\circ d_2]^*])\neq S=\PLab([d_2\circ t\circ d_2^*])$. Hence $d_1\cdot [d_2\circ t\circ d_2^*]=0$. Thus $d_1\cdot(d_2\cdot t)=0$.
Therefore the zero case is done.

Now we consider the nonzero case. Assume that $\PLab([d_2\circ t\circ d_2^*])=S$ and $\PLab([[d_1\circ d_2]\circ t\circ [d_1 \circ d_2]^*])=S$.
Let $s=[d_2\circ t\circ d_2^*]$ and let $d_3=[d_1\circ d_2]$. Then $s\in I_n$ with $\PLab(s)=S$ by assumption. Therefore,
\begin{equation}\label{1}
d_1\cdot(d_2\cdot t)
=
\lambda(d_2,t) \lambda(d_1,s)
[d_1\circ d_2\circ t\circ d_2^*\circ d_1^*].
\end{equation}
On the other hand, since $d_1d_2=\lambda(d_1,d_2)d_3$, we get
\begin{equation}\label{2}
(d_1d_2)\cdot t
=
\lambda(d_1,d_2) \lambda(d_3,t)
[d_3\circ t\circ d_3^*].
\end{equation}
Since $d_3=[d_1\circ d_2]$, we have $[d_3\circ t\circ d_3^*]=[d_1\circ d_2\circ t\circ d_2^*\circ d_1^*]$. Hence
\[
(d_1d_2)\cdot t
=
\lambda(d_1,d_2) \lambda(d_3,t)
[d_1\circ d_2\circ t\circ d_2^*\circ d_1^*].
\]
Since the Okada algebra $O_n(X,Y)$ is associative, we have $d_1(d_2t)=(d_1d_2)t$ implying $\lambda(d_1,[d_2 \circ t])\lambda(d_2,t) = \lambda(d_1,d_2)\lambda(d_3,t)$. By Lemma \ref{halves}, $| [d_2 \circ t] \rangle = | [d_2 \circ t \circ d_2^*] \rangle$, so $\lambda(d_1, [d_2 \circ t])= \lambda(d_1, [d_2 \circ t \circ d_2^*])$ by Lemma \ref{properties of lambda}(1). Therefore the scalar factors on both sides of Equations \eqref{1} and \eqref{2} are equal.
Thus $d_1\cdot(d_2\cdot t)=(d_1d_2)\cdot t$. Hence the nonzero case is done.
\end{proof}
Part (2) of Theorem \ref{main} is proved in Section \ref{gelfand section}. 
\begin{remark}
Note that Gelfand models for the symmetric group $S_n$ have been constructed on the space of involutions in $S_n$ in \cite{KodiyalamVerma2004} and \cite{AdinPostnikovRoichman2008}. The group algebra of $S_n$ is realized as a diagram algebra and a Gelfand model on the space of symmetric diagrams is visualized in terms of a conjugation action in \cite[Section 2 and Remarks 4.2 and 4.4]{HalversonReeks}. Thanks to \cite[Proposition~4.13, Theorem~4.14]{HivertScott}, we have a bijection between paths in the Young--Fibonacci lattice starting from $1$ and ending at a vertex of rank $n$ and the set $I_n$ of symmetric Okada arc diagrams of rank $n$. Therefore, the Gelfand model $M_n$ defined in Theorem \ref{main} can be realized as a model on the space of paths in $\YF$ from $1$ to some $v \in \YF_n$.
\end{remark}
\begin{definition}
Let $S$ be a Fibonacci set of rank $n$. Define
\[
I_n^S=\{t\in\mathcal{D}_n : t=t^*,\ \PLab(t)=S\},
\]
and set $M_n^S=KI_n^S$.
\end{definition}
\begin{remark}\label{nonempty}
Note that $I_n^S \neq \emptyset$, and hence $M_n^S \neq 0$, for all $S \in \YF_n$ by \cite[Lemma 5.15]{HivertScott} and the gluing lemma \cite[Lemma 4.4]{HivertScott}. Moreover $M_n= \bigoplus_{\substack{S \in \YF_n}}M_n^S$, and by the definition of the module action \eqref{action} each $M_n^S$ is an $O_n(X,Y)$-module.
\end{remark}

\begin{example}\label{2 case}
  This example constructs a Gelfand model for $O_2=O_2(X,Y)$. Recall that $O_2$ is
the two-dimensional algebra over $K$ generated by $E_1$, subject
to the relation
$
E_1^2=x_1E_1.$
Hence $
O_2\cong K[z]/\langle z(z-x_1)\rangle$ as algebras.
Since the ideals $\langle z\rangle$ and $\langle z-x_1\rangle$ are
comaximal, the Chinese remainder theorem gives
\[
O_2\cong
K[z]/\langle z\rangle
\times
K[z]/\langle z-x_1\rangle
\cong K\times K.
\]
Thus the algebra $O_2$ has two irreducible representations, up to isomorphism, each of dimension $1$.

We have two Fibonacci sets of rank $2$, namely $S=\{1,2\}$ and $S=\emptyset$. Then we have the representations $M_2^S$ corresponding to these two Fibonacci sets, and both of them are one-dimensional. Hence they are irreducible.
For $S=\{1,2\}$, the module $M_2^{\{1,2\}}$ has basis $\{\mathrm{id}_2\}$. For $S=\emptyset$, the module $M_2^\emptyset$ has basis $\{G_1\}$.
We now show that these two modules are not isomorphic. Suppose, for contradiction, that they are isomorphic. Let $\Phi:M_2^\emptyset\to M_2^{\{1,2\}}$ be an isomorphism. Then $\Phi(G_1)=k\cdot \mathrm{id}_2$ for some nonzero scalar $k$.
Now $\Phi(G_1\cdot G_1)=\Phi(x_1G_1)=kx_1\cdot \mathrm{id}_2$. On the other hand, since $\Phi$ is a module homomorphism, we have $\Phi(G_1\cdot G_1)=G_1\cdot \Phi(G_1)=k(G_1\cdot \mathrm{id}_2)$. But $G_1\cdot \mathrm{id}_2=0$, since $G_1^*=G_1$ and $\PLab([G_1\circ \mathrm{id}_2\circ G_1])=\emptyset\neq \{1,2\}=\PLab(\mathrm{id}_2)$.
Thus $kx_1\cdot \mathrm{id}_2=0$. Since $k\neq 0$ and $x_1\neq 0$, this is a contradiction. Therefore $M_2^\emptyset$ and $M_2^{\{1,2\}}$ are not isomorphic.
Thus the two modules $M_2^\emptyset$ and $M_2^{\{1,2\}}$ are irreducible and non-isomorphic. Since $O_2$ has only two irreducible representations up to isomorphism,
these exhaust all the irreducible representations of $O_2$.
\end{example}

\section{The Jones basic construction for Okada algebras}
In this section, we use the Jones basic construction to describe the Gelfand model for \(O_n\) inductively. We identify suitable idempotents \(e_n\in O_n\) and consider the ideals
$J_n=O_ne_nO_n.$
The complementary subalgebras \(C_n\) are isomorphic as algebras to the Okada algebras \(O_{n-1}\), while double centralizer theory relates the representation theory of \(J_n\) to that of \(O_{n-2}\). This allows us to construct the model representation for \(O_n\) recursively from the model representations of the smaller Okada algebras.
\subsection{The idempotents and the conditional expectation}
For $n\geq 3$, set
\[
e_n:=\frac{1}{x_{n-1}}\,G_{n-1}\in O_n(X,Y).
\]

\begin{lemma}\label{jones lemma}\leavevmode
Let $n\geq 3$.
\begin{enumerate}
\item $e_n^2=e_n$, and $ae_n=e_na$ for all $a\in O_{n-2}$.
\item For $a\in O_{n-2}$, $aG_{n-1}=0$ implies $a=0$.
\item $O_{n-2}e_n=e_nO_ne_n$.
\end{enumerate}
\end{lemma}

\begin{proof}
(1). Since $G_{n-1}G_{n-1}=x_{n-1}G_{n-1}$, we have
$e_n^2=x_{n-1}^{-2}G_{n-1}G_{n-1}=x_{n-1}^{-1}G_{n-1}=e_n$.
The algebra $O_{n-2}$ is generated by $G_1,\ldots,G_{n-3}$, and
$|i-(n-1)|\geq 2$ for $1\leq i\leq n-3$, so $G_iG_{n-1}=G_{n-1}G_i$ for all
such $i$. Hence $ae_n=e_na$ for all $a\in O_{n-2}$.
\vspace{2mm}
\\\noindent
(2). Let $a$ be a basis Okada arc diagram of rank $n-2$. We regard $a$ as an
element of $O_n$ by adding the two propagating arcs
$h(n-1\arcjoin\overline{n-1})=n-1$ and $h(n\arcjoin\overline n)=n$. Then
$aG_{n-1}=\lambda(a,G_{n-1})[a\circ G_{n-1}]$ with $\lambda(a,G_{n-1})\neq0$,
and $[a\circ G_{n-1}]$ is the basis Okada arc diagram obtained by replacing
these propagating arcs by the non-propagating arcs $h(n-1\arcjoin n)=n-1$ and
$h(\overline{n-1}\arcjoin\overline n)=n-1$, while all the remaining arcs and
their height labels are unchanged. Thus distinct basis diagrams of $O_{n-2}$
give distinct diagrams after multiplication by $G_{n-1}$, since the original
diagram is recovered by removing those two non-propagating arcs. Hence their
images are linearly independent, and therefore $aG_{n-1}=0$ implies $a=0$ for
all $a\in O_{n-2}$.
\vspace{2mm}
\\\noindent
(3). Let $d\in O_{n-2}$. By Lemma~\ref{jones lemma}(1), $de_n=de_n^2=e_nde_n\in e_nO_ne_n$, so
$O_{n-2}e_n\subseteq e_nO_ne_n$. Conversely, let $d$ be a basis Okada arc
diagram in $O_n$. Then $[G_{n-1}\circ d\circ G_{n-1}]$ is an Okada arc diagram
containing the two non-propagating arcs $h(n-1\arcjoin n)=n-1$ and
$h(\overline{n-1}\arcjoin\overline n)=n-1$. Removing these arcs and replacing
them by the propagating arcs $h(n-1\arcjoin\overline{n-1})=n-1$ and
$h(n\arcjoin\overline n)=n$ gives an element $a$ of $O_{n-2}$, and
$[G_{n-1}\circ d\circ G_{n-1}]=[a\circ G_{n-1}]$. Hence
$e_nO_ne_n\subseteq O_{n-2}e_n$.
\end{proof}
Recall that for generic $X,Y$, $O_1(X,Y) \subseteq O_2(X,Y) \subseteq \cdots $ forms a tower of finite-dimensional split semisimple algebras with 1 over $K$. Now we record the following properties of the idempotent $e_n$.
\begin{proposition}\label{idempotent}
   The algebras $O_{n-2}$ and $O_{n-2}e_n$ are isomorphic via the map $a \mapsto a e_n$ where $a \in O_{n-2}$.
\end{proposition}
\begin{proof}
 Consider the linear map
$\varphi:O_{n-2}\longrightarrow O_{n-2}e_n,\quad \varphi(a)=ae_n.$ The map $\varphi$ is an algebra homomorphism, since
$(ae_n)(be_n)=abe_n=\varphi(ab)$ for $a, b \in O_{n-2}$ by Lemma~\ref{jones lemma}(1).
This map is surjective by the definition of $O_{n-2}e_n$. The map is injective by Lemma \ref{jones lemma}(2). Hence $\varphi$ is an isomorphism.
\end{proof}
Since we have $O_{n-2} \cong O_{n-2}e_n = e_nO_ne_n$, via the map $a \mapsto ae_n$ by Lemma \ref{jones lemma}(3) and  Proposition \ref{idempotent}, it makes sense to define the following. 
\begin{definition}\label{conditional expectation}
The map $\epsilon_n\colon O_n \to O_{n-2}$ sending $a \in O_n$ to the unique
element $\epsilon_n(a) \in O_{n-2}$ satisfying $e_nae_n=\epsilon_n(a)e_n$ is
called the \emph{conditional expectation}. This is
\cite[(5.1.2)]{HalversonPartition} with $A=O_{n-2}$, $B=O_n$ and $e=e_n$; the
notion goes back to \cite{Jones}.
\end{definition}
\subsection{The recursive structure of the irreducible representations}
Now we state the double centralizer theorem for split semisimple algebras, which
will be used in this section.
\begin{theorem}[{see \cite[Theorem 1]{Andr}}]\label{double centralizer}
Let $A$ be a finite-dimensional split semisimple algebra over a field $K$ of
characteristic zero and let $V$ be a finite-dimensional faithful $A$-module.
Let $C=\mathrm{End}_A(V)$. Then the following hold.
\begin{enumerate}
\item The image of $A$ in $\mathrm{End}_K(V)$ is equal to $\mathrm{End}_C(V)$.
\item The algebra $C$ is split semisimple.
\item The map $V_i\mapsto \operatorname{Hom}_A(V_i,V)$ gives a bijection between
a complete set of pairwise non-isomorphic irreducible $A$-modules and a complete
set of pairwise non-isomorphic irreducible $C$-modules. In particular, the index
sets $\Lambda_A$ and $\Lambda_C$ for the irreducible representations of $A$ and
of $C$ may be identified.
\end{enumerate}
\end{theorem}
The following proposition together with the double centralizer theorem \ref{double centralizer} gives a recursive description of the irreducible representations of the Okada algebra $O_n$.
\begin{proposition}
     View $O_ne_n$ as a module for $O_ne_nO_n$ by multiplication on the left and as a module for $e_nO_ne_n$ by multiplication on the right. Then
    \[
     J_n:= O_ne_n O_n \cong \mathrm{End}_{e_nO_ne_n}(O_ne_n) ~~\text{and} ~~ O_{n-2} \cong e_nO_ne_n \cong \mathrm{End}_{O_ne_nO_n}(O_ne_n). 
    \] 
\end{proposition}
\begin{proof}
By Lemma~\ref{jones lemma} and Proposition~\ref{idempotent}, the hypotheses of \cite[(5.1.1)]{HalversonPartition} are satisfied with $A=O_{n-2}, B=O_n$ and $e=e_n$, and the result follows exactly as in \cite[Proposition 5.1.3]{HalversonPartition}. Note that Halverson works over the complex numbers, but the proof works for any field $K$ over which the algebra is split semisimple; moreover,  $O_n$ is split semisimple by Proposition \ref{split}.
\end{proof}
Note that $J_n$ is a two-sided ideal of $O_n$, and is semisimple, being an ideal of a semisimple algebra. 
Since we have $O_{n-2} \cong e_n O_n e_n$, it follows that 
\begin{equation}\label{double}
    J_n \cong \mathrm{End}_{O_{n-2}} (O_ne_n) ~~\text{and}~~ O_{n-2}  \cong \mathrm{End}_{J_n}(O_ne_n). 
\end{equation}
\begin{proposition} Let $V$ be the span of all Okada arc diagrams of rank $n$ whose propagating label set does not contain $n$. Then $J_n=V$.
\end{proposition}
\begin{proof}
For Okada basis arc diagrams $d$ and $d'$ of rank $n$, Proposition \ref{structure} gives $\PLab([d\circ G_{n-1}\circ d']) \preccurlyeq \PLab(G_{n-1}),$ and $n \notin \PLab(G_{n-1})$; hence $ n \notin \PLab([d\circ G_{n-1}\circ d']).$ It follows that $O_ne_nO_n \subseteq V$. Conversely, given an arc diagram $d \in V$ (equivalently, the arc diagram does not contain the arc $h(n \arcjoin \overline{n})=n$), we have $d= \alpha_d ~\iota_{n-1}(d') G_{n-1}t $ for some scalar $\alpha_d \neq 0, d' \in \mathcal{D}_{n-1}$ and $t \in O_n$ by Proposition \ref{J_n} and hence $d \in J_n$. 
\end{proof}

Let $C_n$ be the span of the Okada arc diagrams of rank $n$ whose propagating label set contains $n$. Then $C_n= \iota_{n-1}(O_{n-1})$ and  therefore $C_n \cong O_{n-1}$ is a subalgebra of $O_n$ such that 
\begin{equation}\label{decomposition}
O_n=J_n\oplus C_n \qquad \text{as vector spaces and}
\qquad
O_n/J_n\cong_{\mathrm{alg}}C_n
\cong_{\mathrm{alg}}O_{n-1}.
\end{equation}
It follows from the double centralizer theorem \ref{double centralizer} applied to \eqref{double} that $O_{n-2}$ and $J_n$ have the same indexing set of irreducible modules. From \eqref{decomposition}, we have 
\begin{equation}\label{recursion}
\Lambda_{O_n}= \Lambda_{J_n} \sqcup \Lambda_{C_n}= \Lambda_{O_{n-2}}\sqcup \Lambda_{O_{n-1}} 
    \end{equation}
where $\Lambda_A$ denotes the index set for the irreducible representations of the algebra $A$. The recursive application of \eqref{recursion} yields  
\begin{equation}
    \Lambda_{O_n}=\begin{cases}
     \Lambda_{O_{n-1}} \sqcup \Lambda_{O_{n-3}}\sqcup\cdots \sqcup \Lambda_{O_3}\sqcup  \Lambda_{O_2}& \text{if $n$ is even},\\
      \Lambda_{O_{n-1}} \sqcup \Lambda_{O_{n-3}}\sqcup \cdots \sqcup \Lambda_{O_2} \sqcup \Lambda_{O_1} & \text{if $n$ is odd}.
    \end{cases} 
\end{equation}
This suggests that the irreducible representations of the Okada algebra could be constructed inductively.

\subsection{Character relations}
Now we record some results concerning the characters of these representations. The exposition mainly follows Section 4 of \cite{HalversonReeks}. 
\begin{definition}
    Let $A$ be a finite-dimensional algebra over $K$ and let $W$ be a finite-dimensional $A$-module, with corresponding representation $\rho: A \to \operatorname{End}_K(W)$. The \emph{character} of $W$ is the $K$-linear map $\chi: A \to K$ defined by $\chi(a)= \operatorname{tr}(\rho(a)).$
\end{definition}

Let $a \in J_n$ with $a= a_1e_na_2$ for some $a_1,a_2 \in O_n$. Let $\chi$ be a character of a representation of $O_n$ and $\epsilon_n$ the conditional expectation (Definition \ref{conditional expectation}). Then by the property of trace, 
\[
\chi(a)= \chi(a_1e_na_2)=\chi(a_2a_1e_n^2)=\chi(e_n(a_2a_1)e_n)=\chi(\epsilon_n(a_2a_1)e_n).
\]
Thus it follows from the above calculation and \eqref{decomposition} that the character $\chi$ is completely determined by its values on the elements of $C_n$ and $O_{n-2}e_n$.
\begin{proposition}\label{char}
    For $\lambda \in \Lambda_{O_n}$ and $a \in O_{n-2}$, the corresponding character $\chi_{O_n}^\lambda$ satisfies
\begin{equation}
    \chi_{O_n}^\lambda(ae_n) = \begin{cases}
    \chi_{O_{n-2}}^\lambda(a) & \text{if } \lambda \in \Lambda_{O_{n-2}}, \\
    0 & \text{if } \lambda \in \Lambda_{O_n} \setminus \Lambda_{O_{n-2}}. 
    \end{cases}
\end{equation}
\end{proposition}
\begin{proof}
   Halverson proves the analogous result for finite-dimensional semisimple algebras over the complex numbers \cite[Proposition 5.1.9]{HalversonPartition}. But the same argument carries over to any finite-dimensional split semisimple algebra over any field. Since $O_n$ is split semisimple by Proposition \ref{split}, the result follows. 
\end{proof}
 
 Then from \eqref{decomposition} and Proposition \ref{char},
\begin{equation}\label{bridge}
\chi_{O_n}^\lambda(a)=
\begin{cases}
\chi_{C_n}^\lambda(a), & \text{if $\lambda \in \Lambda_{C_n}$ and $a\in C_n$},\\
0, & \text{if $\lambda \in \Lambda_{C_n}$ and $a\in J_n$},\\
\chi_{O_{n-2}}^\lambda(a'), & \text{if $\lambda \in \Lambda_{O_{n-2}}$ and $a=a'e_n$ with $a'\in O_{n-2}$}.
\end{cases}
\end{equation}
The computation of the character values $\chi_{O_n}^\lambda(a)$ for $\lambda \in \Lambda_{O_{n-2}}, a \in C_n$, is more involved and we do not need them. The relations \eqref{bridge} are what allow the argument in the next section to be carried out at the level of characters.
\section{Gelfand model from the Jones basic construction}\label{gelfand section}
This section is devoted to the proof of part~(2) of our main Theorem~\ref{main}. We use Theorem~\ref{gelfand model} to prove that the module $M_n$ is indeed a Gelfand model. We prove this by induction on $n$. We assume that $M_m$ is a Gelfand model for $O_m(X,Y)$ for all $m < n$. The base cases are $m=1$, which is trivial, and $m=2$, done in Example \ref{2 case}. 
\subsection{A Gelfand model criterion}
To proceed further, we construct a collection of submodules
$\{M_n^r\mid 0\leq r\leq n\}$ of $M_n$ such that
$M_n=\bigoplus_{r=0}^n M_n^r$, satisfying the following conditions (A) and (B).
For each integer $0\leq r\leq n$, define
\[
M_n^r=\bigoplus_{\substack{S\in\YF_n\\ \max S=r}}M_n^S
\]
where we use the convention that $\max\emptyset=0$.
Then $M_n^r\neq0$ only if $r\equiv n\pmod 2$.
\begin{enumerate}[label=(\Alph*)]
    \item The character $\phi_n^r$ of $M_n^r$ satisfies
\begin{equation}\label{character relations}
\phi_n^r(a)=
\begin{cases}
\phi_{C_n}(a), & \text{if $r=n$ and $a\in C_n$},\\
0, & \text{if $r=n$ and $a\in J_n$},\\
\phi_{n-2}^r(a'), & \text{if $r<n$ and $a=a'e_n$ with $a'\in O_{n-2}$}.
\end{cases}
\end{equation}
Here $\phi_{C_n}$ denotes the character of the Gelfand model $M_{n-1}$ of $C_n \cong O_{n-1}$.
   \vspace{2mm} 
   \item The modules $M_n^r$ and $M_n^s$ have no common irreducible constituents whenever $r\neq s$. 
\end{enumerate} 
\begin{theorem}\label{gelfand model}
Under the Jones basic construction described in the previous section, if the submodules $M_n^r$ with
$M_n=\bigoplus_{r=0}^n M_n^r$ satisfy (A) and (B), then $M_n$ is a Gelfand model for $O_n(X,Y)$, i.e., each irreducible representation of $O_n(X,Y)$ occurs exactly once in the decomposition of $M_n$.
\end{theorem}
\begin{proof}
    The proof follows from the arguments of \cite[Theorem 4.1]{HalversonReeks}. 
     Halverson and Reeks work over an algebraically closed field, but their proof of \cite[Theorem 4.1]{HalversonReeks} carries over to finite-dimensional split semisimple algebras over any field. Indeed, they prove it at the level of characters; i.e., they prove that the character of the module $M_n$ is the sum of characters of the irreducible representations of $O_n$ using the analogue of \eqref{bridge}. The irreducible characters of a split semisimple algebra are linearly independent; hence the argument applies verbatim in our setting. 
\end{proof}
\subsection{Proof of part (2) of Theorem~\ref{main}}
Thus, in order to prove that $M_n$ is a Gelfand model for the Okada algebra $O_n(X,Y)$, we need to show that the submodules $M_n^r$ satisfy (A) and (B). We establish this below in Propositions~\ref{A} and \ref{B}. 

\begin{proposition}\label{A}
The character $\phi_n^r$ of $M_n^r$ satisfies equation \eqref{character relations}.
\end{proposition}

\begin{proof}
Let $b$ be an Okada arc diagram in $J_n$, so that
$n\notin \PLab(b)$.
Consider a basis element $t\in M_n^n$. Then
$n\in \PLab(t)$.
By Proposition~\ref{structure},
$\PLab([b\circ t\circ b^*])
\preccurlyeq \PLab(b)$.
Since $n\notin \PLab(b)$, we have
$n\notin \PLab([b\circ t\circ b^*])$.
Therefore,
$\PLab([b\circ t\circ b^*])
\neq \PLab(t)$,
and hence
$b\cdot t=0$.
Thus every basis diagram in $J_n$ acts as zero on $M_n^n$. By linearity,
\begin{equation*}
\phi_n^n(a)=0
\quad
\text{for every }a\in J_n.
\end{equation*}

By the induction hypothesis, $M_{n-1}$ is a Gelfand model for $O_{n-1}$.
Let $b\in C_n \cong O_{n-1}$ be a basis arc diagram. Since
$n\in \PLab(b)$,
the diagram $b$ contains the propagating arc $h(n \arcjoin \overline{n})=n$ and therefore $b$ can be regarded as an element of $O_{n-1}$. Let $t\in M_n^n$ be a basis element. Since
$n\in \PLab(t)$,
the diagram $t$ contains the propagating arc $h(n \arcjoin \overline{n})=n$. Removing this arc gives a symmetric Okada arc diagram $t'$ of rank $n-1$ and thus there is a bijection between the basis of the Gelfand model $M_{n-1}$ and the basis of $M_n^n$ given by the map $t'\mapsto t= \iota_{n-1}(t')$. Thus the basis of $M_n^n$ can be identified with the basis of the Gelfand model $M_{n-1}$ for $O_{n-1}$.
Under this identification, the action of $b$ on $t'$ in the Gelfand model $M_{n-1}$ for $O_{n-1}$ is the same as the action of $b$ on $t$ in the module $M_n^n$ for $O_n$. Therefore, their character values are equal, and hence by linearity,
\begin{equation*}
\phi_n^n(a)=\phi_{C_n}(a) \quad \text{for every }a\in C_n.
\end{equation*}

Let $a=a'e_n$ for a basis Okada arc diagram $a'\in O_{n-2}$ and
$e_n=x_{n-1}^{-1}G_{n-1}$. Since $a'$ commutes with $e_n$ and $\lambda(a', G_{n-1})=1$ by Lemma \ref{properties of lambda}(2), we have
$a=e_na'=x_{n-1}^{-1}[G_{n-1}\circ a']$.
Let $t$ be a basis element of $M_n^r$ with $r<n$. Let
$b=[a'\circ G_{n-1}]$. 
The basis element $t$ contributes to the trace of $a$ only if
$[b\circ t\circ b^*]=t$.
Since $b=[G_{n-1}\circ a']$ and
$b^*=[a'^*\circ G_{n-1}]$, we have
$[b\circ t\circ b^*]\in e_nO_ne_n$ and we already know that $e_nO_ne_n=O_{n-2}e_n$.
Hence, if $[b\circ t\circ b^*]=t$, then $t$ must have the form
$t=[t'\circ G_{n-1}]$ for a unique basis element $t'$ of $M_{n-2}^r$. Assume $[b\circ t\circ b^*]=t$ so that $\PLab([b\circ t\circ b^*])=\PLab(t)$. 
Now using Lemma \ref{jones lemma}(1) and the associativity of the Okada monoid, we have
\begin{equation*}
\begin{aligned}
a\cdot t
&= x_{n-1}^{-1}\lambda(b,t)
[[a'\circ G_{n-1}]\circ [t'\circ G_{n-1}]\circ [a'^*\circ G_{n-1}]]\\
&= x_{n-1}^{-1} \lambda(b,t)
[a'\circ t'\circ a'^*\circ G_{n-1}].
\end{aligned}
\end{equation*}
We have
\[
at =x_{n-1}^{-1} bt = x_{n-1}^{-1} \lambda(b,t)[b \circ t]= x_{n-1}^{-1}\lambda(b,t)[a' \circ t' \circ G_{n-1}] 
\] and 
\begin{align*}
at &= x_{n-1}^{-1}\,(a'G_{n-1})(t'G_{n-1})
   && \text{since } \lambda(a',G_{n-1})=1=\lambda(t',G_{n-1}) \text{ by Lemma \ref{properties of lambda}(2)},\\
   &= x_{n-1}^{-1}\,a't'G_{n-1}G_{n-1}
   && \text{by Lemma~\ref{jones lemma}(1)},\\
   &= a't'G_{n-1}
   && \text{since } G_{n-1}G_{n-1}=x_{n-1}G_{n-1},\\
   &= \lambda(a',t')\,[a'\circ t'\circ G_{n-1}]
   && \text{since } \lambda([a'\circ t'],G_{n-1})=1
      \text{ by Lemma~\ref{properties of lambda}(2)}.
\end{align*}
 Hence we have 
$x_{n-1}^{-1} \lambda(b,t)=\lambda(a',t')$.
Thus
\begin{equation*}
a\cdot t=
\lambda(a',t')
[a'\circ t'\circ a'^*\circ G_{n-1}].
\end{equation*}
If $t$ is not of the form $[t'\circ G_{n-1}]$ then the $t$-$t$ entry of $a$ on $M_n^r$ is $0$. If $t=[t'\circ G_{n-1}]$, then that entry equals the $t'$-$t'$ entry of the action of $a'$ on $M_{n-2}^r$: both are 
$\lambda(a',t')$ when $[a' \circ t' \circ a'^*]=t'$, and both are $0$ otherwise.
Since we have the bijection
$t=[t'\circ G_{n-1}]\leftrightarrow t'$
between the basis elements contributing to the two traces, we obtain
\begin{equation*}
\phi_n^r(a'e_n)=\phi_{n-2}^r(a').
\end{equation*} By linearity, the above equation is true for all $a \in O_n $ with $a=a'e_n$ for some $a' \in O_{n-2} $.
\end{proof}
\begin{proposition}\label{B}
Let $S,T\in \YF_n$. Then $\operatorname{Hom}_{O_n}(M_n^S,M_n^T) \neq 0$ implies $T \preccurlyeq S$. In particular $M_n^S$ and $M_n^T$ do not have any common irreducible constituents for $S \neq T$.
\end{proposition}
\begin{proof}
Suppose there is a nonzero $O_n$-module homomorphism
$\phi:M_n^S\to M_n^T$. Choose $t\in I_n^S$ such that
$\phi(t)\neq0$. Since $t=t^*$, we have
$t\cdot t=\alpha_t t$ for some nonzero scalar $\alpha_t$. Hence
$\phi(t\cdot t)=\alpha_t\phi(t)\neq0$.
Since $\phi$ is a module homomorphism,
$\phi(t\cdot t)=t\cdot\phi(t)$.
Write
$\phi(t)=\sum_{u\in I_n^T}a_u u$.
Since $t\cdot\phi(t)\neq0$, there exists some $u\in I_n^T$ with
$a_u\neq0$ such that $t\cdot u\neq0$. This implies
$\PLab([t\circ u\circ t^*])=T$.
Since $t=t^*$,
$T=\PLab([t\circ u\circ t])$.
By Proposition~\ref{structure},
$T\preccurlyeq\PLab(t)=S$.
Thus, if there is a nonzero module homomorphism from $M_n^S$ to $M_n^T$, then
$T\preccurlyeq S$.

Since $O_n$ is semisimple, if $M_n^S$ and $M_n^T$ have a common irreducible constituent, then there are nonzero $O_n$-module homomorphisms in both directions. Hence
$T\preccurlyeq S$ and $S\preccurlyeq T$, which implies $T=S$.
\end{proof}
It follows from Proposition \ref{B} that $\operatorname{Hom}_{O_n}(M_n^r,M_n^s)=0$ if $r \neq s$ for $0 \leq r,s \leq n$ and therefore they do not have any common irreducible constituents. This together with Proposition~\ref{A} completes the proof of Theorem~\ref{main}.

\section{Irreducible modules for the Okada algebra}

We have constructed a representation for the Okada algebra $O_n$ which
contains each irreducible representation exactly once. We use this model
to obtain all the irreducible representations of the algebra. We then establish the isomorphism of these irreducible modules with the cell modules defined in \cite{HivertScott} and conclude that the Bratteli
diagram of the tower of Okada algebras $\{O_n\}_{n\geq 1}$ is the
Young--Fibonacci lattice.

\begin{theorem}
The Okada modules $M_n^S$, for all Fibonacci sets $S$ of rank $n$,
constitute a complete list of irreducible modules for $O_n(X,Y)$.
\end{theorem}

\begin{proof}
Note that by \eqref{recursion}, $\Lambda_{O_n} = \Lambda_{O_{n-1}} \sqcup \Lambda_{O_{n-2}} $ and we already showed that the number of pairwise non-isomorphic irreducible representations for $O_1$ and $O_2$ is one and two, respectively. Thus it follows that the number of irreducible representations of $O_n$ is the $n$th Fibonacci number $F_n$. Since $M_n= \bigoplus_{S \in \YF_n}M_n^S$ is a Gelfand model, it has exactly $F_n$ irreducible constituents. By Proposition~\ref{B}, $M_n^S$ and $M_n^T$ do not contain any isomorphic irreducible submodule for $S\neq T$, and $M_n^S \neq 0$ for every  $S \in \YF_n $ by Remark \ref{nonempty}. Since the cardinality of $\YF_n$ is $F_n$, each $M_n^S$ is an irreducible module, and they are pairwise non-isomorphic. 
\end{proof}
\begin{remark}
    We obtain all the irreducible representations of $O_n$ without any prior information about the irreducibles beyond their dimensions, which are used in Proposition \ref{split} in showing that the algebra $O_n$ is split semisimple. If the splitness of the algebra could be established without appealing to any information about the irreducibles, this would give a completely independent construction of the irreducible $O_n$-modules via the Jones basic construction and the diagram realization independent of Okada's original one. Note that even though Okada's construction and ours are realized on the same space of symmetric diagrams (equivalently, on paths by Fomin's correspondence), Okada's action is more involved as the action of a generator on a basis element gives a nontrivial linear combination of the basis elements whereas the action defined in \eqref{action}, being a conjugation action, sends a basis element to a scalar multiple of a single basis element. 
\end{remark}
Hivert and Scott prove that the Okada algebra $O_n(X,Y)$ is cellular in \cite[Theorem 6.3]{HivertScott}.
The left $O_n(X,Y)$ cell module $V^S$ associated to $S \in \YF_n$ is defined on the vector space spanned by rank $n$ half Okada arc diagrams $|t \rangle$ such that $\PLab(|t \rangle)=S$ equipped with the left action $\bullet$ by extending linearly the formula
\begin{equation}\label{Scottmodules}
    d \bullet |t \rangle := \begin{cases}
    \lambda(d,t)| [d \circ t] \rangle \quad &\text{if $\PLab([d \circ t])=S$}\\
    0 \quad & \text{otherwise}
    \end{cases}
\end{equation}
where $d \in \mathcal{D}_n$. Note that the action \eqref{Scottmodules} is independent of the choice of a representative $t$ with $ |t \rangle$ equal to the given half diagram by \cite[Remark 6.4]{HivertScott}.
\begin{proposition}\label{iso}
    The modules $V^S$ and $M_n^S$ are isomorphic as $O_n(X,Y)$-modules for generic $X$ and $Y$.
\end{proposition}
\begin{proof}
    For a symmetric diagram $t \in I_n^S$, $| t \rangle = \langle t |$ and conversely given an Okada half arc diagram $H$ of rank $n$, there is a unique symmetric diagram $t \in I_n^S$ with $ | t \rangle =H=  \langle t |$ and $\PLab(t)=S$ by the gluing lemma \cite[Lemma 4.4]{HivertScott}.
Hence the map $\phi: M_n^S \to V^S$ defined by $\phi(t)= | t \rangle$ extends to a linear bijection.

Recall that for a symmetric diagram $t \in I_n^S$ and $d \in \mathcal{D}_n$, $\PLab([ d \circ t \circ d^*])= \PLab(t)=S$ if and only if $\PLab([d \circ t])=S$ by Lemma \ref{halves}.
Therefore $\PLab([d \circ t]) \neq S$ if and only if $\PLab([d \circ t \circ d^*]) \neq S$; hence $d \bullet \phi(t)=0$  if and only if $d \cdot t =0$.
In the remaining case, $\PLab([d \circ t]) = \PLab([ d \circ t \circ d^*])=S$ and
\[
\phi(d \cdot t)= \lambda(d,t) | [ d \circ t \circ d^*] \rangle = \lambda(d,t) |[d \circ t]\rangle = d \bullet \phi(t).
\]
Hence the map $\phi$ is an isomorphism of $O_n(X,Y)$-modules. 
\end{proof}
\begin{remark}
For $S\in\YF_n$, let $v$ be the word in the alphabet $\{1,2\}$ corresponding to
$S$ under the bijection of Remark~\ref{YF equivalence}. By
Proposition~\ref{iso}, we have $M_n^S\cong V^S$, and the cell module $V^S$ is
isomorphic to the irreducible $O_n(X,Y)$-module $V_v$ constructed by Okada
in \cite{Okada}; this identification is made in \cite[\S7.1]{HivertScott}.
\end{remark}
\begin{remark}
Recall that the \emph{Bratteli diagram} of the tower of Okada algebras
$\{O_n(X,Y)\}_{n\geq 1}$ is the graph with vertices
$\bigsqcup_{r\geq 1}\Lambda_{O_r}$ and $k$ edges between
$\lambda\in\Lambda_{O_n}$ and $\mu\in\Lambda_{O_{n-1}}$ if the restriction
of the irreducible representation $V^\lambda$ of $O_n$ to $O_{n-1}$
contains $k$ copies of the irreducible representation $V^\mu$ of
$O_{n-1}$ in its decomposition. Then \cite[Proposition 6.6]{HivertScott} together with the isomorphism in Proposition \ref{iso} proves that the Bratteli diagram of the tower of Okada algebras $\{O_n(X,Y)\}_{n\geq 1}$ is the Young--Fibonacci lattice. 
\end{remark}
\section*{Acknowledgements}
The author thanks Sankaran Viswanath for suggesting the problem and for many helpful discussions. The author also thanks Jeanne Scott and Florent Hivert for clarifications regarding their work and for sending the preprint \cite{HivertScott}, and Arun Ram for helpful clarifications regarding \cite{HalversonRam}.
\section*{Declaration on the use of AI tools}
In preparing this manuscript, the author used ChatGPT (OpenAI) and Claude (Anthropic) to polish sentences, to check references and to identify typographical and expository errors. No mathematical content was generated by AI tools. The author reviewed and verified the entire manuscript and takes full responsibility for its contents.
\bibliographystyle{amsalpha}
\bibliography{references}

\end{document}